\documentclass%[12pt]
{amsart}

\usepackage{amssymb,amscd}
\usepackage[all,line,arc,curve,color,frame,pdf]{xy}
\usepackage{tikz} 
\usepackage{tikz-cd}
\usetikzlibrary{positioning, trees, snakes}
\usepackage[textsize=tiny]{todonotes}
\usepackage{hyperref}
\usepackage[shortlabels]{enumitem}
\usepackage{float} 
\usepackage[paper=a4paper, margin=3.5cm]{geometry}
\usepackage{cleveref}  
\usepackage[symbol]{footmisc}
\usepackage{comment}
\usepackage{scalerel}
\usepackage[normalem]{ulem}
\usepackage{booktabs}

\numberwithin{equation}{section}
\theoremstyle{plain}
\newtheorem{thm}{Theorem}[section]
\newtheorem{cor}[thm]{Corollary}
\newtheorem{lemma}[thm]{Lemma}
\newtheorem{prop}[thm]{Proposition}

\theoremstyle{remark}
\newtheorem{rem}[thm]{Remark}
\newtheorem{ex}[thm]{Example}

\theoremstyle{definition}
\newtheorem{defi}[thm]{Definition}

\def\ot{\otimes}

\DeclareMathOperator{\Mat}{Mat}

\DeclareMathOperator{\COL}{COL}
\DeclareMathOperator{\Pic}{Pic}

\DeclareMathOperator\Diag{Diag}

\usepackage[textsize=tiny]{todonotes}

\usepackage{enumitem}

\newcommand\ignore[1]{}

\newcommand\CC{{\mathbb{C}}}
\newcommand\PP{{\mathbb{P}}}
\newcommand\RR{{\mathbb{R}}}
\newcommand\QQ{{\mathbb{Q}}}

\def\operatorname#1{\mathop{\rm #1}\nolimits}

\def\Pic{\operatorname{Pic}}

\def\det{\operatorname{det}}

\newcommand{\Chi}{\ensuremath \raisebox{2pt}{$\chi$}}
\newcommand{\pb}{\ar@{}[dr]|{\text{\pigpenfont J}}}

\makeatletter
\newcommand{\xleftrightarrow}[2][]{\ext@arrow 3359\leftrightarrowfill@{#1}{#2}}
\makeatother
\newcommand{\xdasharrow}[2][->]{
\tikz[baseline=-\the\dimexpr\fontdimen22\textfont2\relax]{
\node[anchor=south,font=\scriptsize, inner ysep=1.5pt,outer xsep=2.2pt](x){#2};
\draw[shorten <=3.4pt,shorten >=3.4pt,dashed,#1](x.south west)--(x.south east);
}}

\newcommand\m{{\mathfrak m}}

\begin{document}
\title{Characteristic numbers and chromatic polynomial of a tensor}

%please add authors and all the data depending on the type of document
\author[Conner]{Austin Conner}
\address{Department of Mathematics, Harvard University, Cambridge, MA 02138}
\email{aconner@math.harvard.edu}

\author[Micha{\l}ek]{Mateusz Micha{\l}ek}
\address{
	University of Konstanz, Germany, Fachbereich Mathematik und Statistik, Fach D 197
	D-78457 Konstanz, Germany
}
\email{mateusz.michalek@uni-konstanz.de}

\begin{abstract}
We introduce the characteristic numbers and the chromatic polynomial of a tensor. Our approach generalizes and unifies the chromatic polynomial of a graph and of a matroid, characteristic numbers of quadrics in Schubert calculus, Betti numbers of complements of hyperplane arrangements and Euler characteristic of complements of determinantal hypersurfaces and the maximum likelihood degree for general linear concentration models in algebraic statistics.
\end{abstract}

\thanks{AC is supported by the NSF grant nr 2002149. MM is supported by the DFG grant nr 467575307}
\maketitle
%\today
%\par\bigskip\hrule\par\bigskip
% contents input of sectio ns
%Make the following sections and divide the structure of the input, make respective \TeX files :
%\tableofcontents

%proposed division in sections and respective tex input files
\section{Introduction}
To a tensor $T\in \CC^a\ot \CC^n\ot\CC^n$ we naturally associate a homogeneous polynomial $P_T$ of degree $a$ in $n-1$ variables. We call the coefficients of $P_T$, which are always integers, the \emph{characteristic numbers} of $T$. Two of the variables of $P_T$ play a special role. Setting all of the other variables to zero we obtain a bivariate homogeneous polynomial $\Chi_T$ that we call the \emph{chromatic polynomial}. Using a related construction, we also define the \emph{relative chromatic polynomial} $\Chi'_T$ and give conditions when $\Chi_T=\Chi'_T$. 

Our setting generalizes many of the important invariants in mathematics.
\begin{itemize}
\item The reduced chromatic polynomial of a representable matroid. This is the case when the contraction $\CC^a(T)$ is a space of simultaneously diagonalizable matrices.
\item The reduced chromatic polynomial of a graph. This is the case when the contraction $\CC^a(T)$ is the row space of the adjacency matrix of the graph.
\item Characteristic numbers of linear systems of quadrics. More precisely the number of quadrics in a linear system that pass through a given number of general points and are tangent to a given number of general hyperplanes. This is the case when the contraction $\CC^a(T)$ consists of symmetric matrices.
\item Maximum likelihood degree of general linear concentration models. This is the case when the contraction $\CC^a(T)$ consists of general symmetric matrices.
\item Euler characteristic of determinantal hypersurfaces. This case arises for arbitrary $T$. However, for special $T$, e.g.~when $\CC^a(T)$ consists of diagonal matrices, we obtain interesting cases, like complements of hyperplane arangements.
\end{itemize}

Our setting is based on the variety of complete collineations $\COL_n$. It is a smooth projective variety with a natural map $\pi_1:\COL_n\rightarrow \PP(\CC^n\otimes \CC^n)$. By contraction, we identify a tensor $T\in \CC^a\ot \CC^n\ot\CC^n$ with the linear space $\CC^a(T)$ of $n\times n$ matrices. By projectivising and taking the strict transform by $\pi_1$ we obtain a subvariety $H_T\subset \COL_n$. The main object of our study is the cohomology class $[H_T]$ associated to $H_T$ in the cohomology ring of $\COL_n$. Intersection product allows us to associate a polynomial function $P_T$ to $[H_T]$ on the Picard group $\Pic_\QQ(\COL_n)$. The vector space $\Pic_\QQ(\COL_n)$ comes with a distinguished basis (in fact two of them), thus we obtain well-defined coefficints of $P_T$. It turns out that these numbers have already appeared in special cases in various disciplines of mathematics  in examples above ranging from combinatorics, through algebraic statistics and topology to intersection theory. As the reader might have noticed, in the examples above we have mentioned diagonal, symmetric and general matrices, however not skew-symmetric matrices. Our construction also works in that case and we believe it may be used, for example to introduce invariants for $p$-groups.

Our article is very much inspired by and based on previous results. The first crucial ones are the seminal papers \cite{June1, June2}. Those were later generalized from representable matroids to nonrepresentable ones \cite{adiprasito2018hodge}, however in our case the geometry of the representable case plays the most important role. This topic may be seen as the `diagonal case' of our constructions. Although, as we point out, many very special things happen in this setting, the generalization to arbitrary tensors goes smoothly. 

The second series of articles relates to algebraic statistics. Here our inspirations are drawn from the maximum likelihood for linear concentration models as described in \cite{StUh}. In the recent papers \cite{ja1, ja2} connections to the cohomology of complete quadrics were made. This topic may be seen as the `symmetric case' of our construction.

Finally, we would like to refer to \cite{June1, dimca2003hypersurface} from where we have drawn the connections to topology.

Our main tools are drawn from results in classical algebraic geometry, especially \cite{Pragacz, LaksovLascouxThorup, ThaddeusComplete, DeConciniProcesi1}.

The plan of the article is as follows. We begin with Section \ref{sec:exm} where we present motivating examples. In Section \ref{sec:def} we present our main construction of characteristic numbers and the chromatic polynomial. We further prove basic theorems about those invariants.  In the final Section \ref{sec:rank} we present a numerical algorithm to compute the chromatic numbers. We also present results how these numbers look like for tensors of fixed rank. As we show our invariants have a potential of distinguishing tensors of high rank.

\section*{Acknowledgements}
We would like to thank Josh Grochow for pointing us towards possible applications in the theory of $p$-groups. We are grateful to Paul Breiding for interesting discussions about numerical algebra and condition numbers, which lead us to a better, more stable version of our algorithms. We thank Rodica Dinu and Tim Seynnaeve for a careful reading of the article and important comments.
 
%general intro
\section{Motivating examples}\label{sec:exm}
In this section we present various examples showing how our construction links different branches of mathematics. At this point a lot may seem a numerical coincidence, however as we will prove the choices of objects (graphs, linear systems of quadrics, statistical models, determinantal hypersurfaces) are arbitrary and always work. In each example we use different fonts to indicate the numbers that are the same. Each case is based on a theorem from the next section.

In the first example we show how our construction generalizes chromatic polynomials of arbitrary graphs.
\begin{ex}\label{ex:graph}
The chromatic polynomial $\chi_G(k)$ of the graph $G$ counts the number of proper vertex colorings with $k$ colors. As long as $G$ has at least one edge we have $\chi_G(1)=0$, thus we define the reduced chromatic polynomial $\bar\chi(k):=\chi(k)/(k-1)$.

Let us fix an arbitrary orientation of all edges of $G$. For simplicity we will assume that $G$ is connected.
We define the adjacency matrix $A_G$ for the graph $G=(V,E)$ as a $|E|\times |V|$ matrix with rows indexed by edges $E$ and columns labelled by vertices $V$ where $A_G(e,v)=1$ if $v$ is the head of $e$, $A_G(e,v)=-1$ if $v$ is the tail of $e$ and $0$ otherwise. The matrix $A_G$ defines a surjection from $\CC^{|E|}$ to the codimension one subspace $C_0\subset \CC^{|V|}$ defined by the condition that the coordinates sum up to zero. Equivalently, the transpose $A_G^t$ defines the injection $C_0^*\subset \CC^{|E|}$.

Let $G$ be the $6$-cycle. We have:
\[\chi_G(k)=(k-1)^6+(k-1),\qquad \bar\chi_G(k)=(k-1)^5+1=k^5-5k^4+\emph{10}k^3-10k^2+{\bf{5}}k.\]

The adjacency matrix is:
\[A_G=
\begin{pmatrix}
-1&1&0&0&0&0\\
0&-1&1&0&0&0\\
0&0&-1&1&0&0\\
0&0&0&-1&1&0\\
0&0&0&0&-1&1\\
1&0&0&0&0&-1\\
\end{pmatrix}
\]
Our aim is to associate to $G$ a linear subspace of diagonal matrices. The image of $A_G^t$ is the five dimensional subspace $C_0^*\subset \CC^{|E|}=\CC^6$ with coordinates summing up to zero. From now on we regard $\CC^{|E|}$ as the space of $6\times 6$ diagonal matrices and $C_0^*$ as the subspace of diagonal matrices with trace zero.

We invert all (invertible) matrices in $C_0^*$ obtaining (after closing) a hypersurface of diagonal matrices of degree ${\bf 5}$, defined by the fifth elementary symmetric polynomial in six variables. If we take a general\footnote{Here `general' means belonging to a Zariski open set in the Grassmannian. Readers not familiar with algebraic geometry may consider a `random' subspace.} subspace of codimension $i$ (e.g.~$i=2$) of $C_0^*$ then the image will be a variety of codimension $i+1$ of the degree equal to the absolute value of the $(i+1)$-st coefficient of $\bar\chi_G$ (e.g.~$\emph{10}$).

In this example the tensor $T\in \CC^5\otimes \CC^6\otimes \CC^6$ and $\CC^5(T)$ is the space of traceless, diagonal $6\times 6$ matrices. We have:
\[\chi_T=\chi_T'=
  a^4
  +5\cdot 4a^3b
  +10\cdot 6 a^2b^2
  +10\cdot 4 ab^3
  +{\bf 5}\cdot b^4
\] 
\end{ex}

In the second example we show how our construction generalizes enumerative problems on quadrics.
\begin{ex}\label{ex:quad}
Consider the following eight dimensional family of degree two polynomials in four variables:
\[ax_0^2+bx_1^2+cx_2^2+dx_3^2+ex_0x_1+fx_1x_2+gx_2x_3+hx_3x_0,\]
where $a,b,c,d,e,f,g,h$ are parameters and $x_0,\dots,x_3$ variables. Each (nonzero) member of this family defines a degree two hypersurface in $\PP^3$. Let us pick seven general hyperplanes in $\PP^3$. There are ${\bf{9}}$ smooth projective quadratic surfaces that are tangent to all the given hyperplanes.

The above family may be identified with the following space of symmetric matrices:
\[
\begin{pmatrix}
a&e&0&h\\
e&b&f&0\\
0&f&c&g\\
h&0&g&d\\
\end{pmatrix}.
\]
Inverting all (invertible) matrices in that space and closing the image we obtain a four dimensional algebraic variety of degree ${\bf{9}}$.

In this example the tensor $T\in \CC^8\otimes \CC^4\otimes \CC^4$ and $\CC^8(T)$ is the space of symmetric matrices given above. We have:
\[\chi_T=
  a^7
  +3\cdot 7 a^6 b
  +9\cdot 21 a^5 b^2
  +17\cdot 33 a^4b^3
  +21\cdot 33 a^3b^4
  +21\cdot 21 a^2b^5
  +17\cdot 7 ab^6
  +{\bf 9} b^7
\] 
\[\chi_T'=
  a^7
  +3\cdot 7 a^6 b
  +9\cdot 21 a^5 b^2
  +17\cdot 35 a^4b^3
  +21\cdot 35 a^3b^4
  +21\cdot 21 a^2b^5
  +15\cdot 7 ab^6
  +5 b^7
\] 
\end{ex}

In the third example we show how our construction generalizes maximum likelihood degree.
\begin{ex}\label{ex:MLE}
Let us consider the following linear space of concentration matrices:
\[
A_{a,b,c,d,e}=\begin{pmatrix}
a&b&d\\
b&c&e\\
d&e&(a+b+c+d+e)\\
\end{pmatrix}.
\]
This means that we consider a family of probability distributions, each one being a multivariate Gaussian distribution on $\RR^3$, parameterized by such $a,b,c,d,e\in \RR$ that $A_{a,b,c,d,e}$ is positive definite and the mean $\mu\in\RR^3$. The associated probability distribution is:
\[f_{a,b,c,d,e,\mu}(x)=\frac{(\det A_{a,b,c,d,e})^{\frac{1}{2}}}{(2\pi)^{\frac{3}{2}}}\exp{\left(-\frac{1}{2}(x-\mu)^TA_{a,b,c,d,e}(x-\mu)\right)}.\]
Such a family of probability distributions is called a statistical model. The one above is known as a linear concentration model \cite{MR0277057}, as it is given by linear conditions on the concentration matrix. One of the aim of statistics is to fit the parameters (in our case $a,b,c,d,e,\mu$) of the model, so that it best explains the given data. The data is a finite family of points $x_1,\dots, x_n\in \RR^3$. First estimating $\mu$ is easy, as one takes the mean of $x_i$'s. For simplicity let us assume that $
\mu=0$, which may be always achieved by shifting the data. 

Our aim is to maximize the likelihood function in parameters $a,b,c,d,e$:
\[\prod_{i=1}^n f_{a,b,c,d,e,0}(x_i).\]
As logarithm is monotonic one considers the log-likelihood function:
\[\sum_{i=1}^n \log f_{a,b,c,d,e,0}(x_i),\]
which, up to a constant equals:
\[\frac{n}{2}\log(A_{a,b,c,d,e})+\sum_{i=1}^n x_i^TA_{a,b,c,d,e}x_i.\]
First one computes the number of complex critical points of the function. By taking the partial derivatives, when $x_i$'s are general, we obtain {\bf two} complex critical points. This number is known as the maximum likelihood degree, which in this case coincides with the degree of the model. Out of the critical points, the maximum we look for will be the unique point for which $A_{a,b,c,d,e}$ is positive definite.

We may take the inverses of all invertible matrices in the space of symmetric matrices specified by the model. We obtain a variety of degree {\bf two}.

In this example the tensor $T\in \CC^5\otimes \CC^3\otimes \CC^3$ and $\CC^5(T)$ is the space of symmetric $3\times 3$ matrices $A_{a,b,c,d,e}$. We have:
\[\chi_T=\chi_T'=
  a^4
  +2\cdot 4a^3b
  +4\cdot 6a^2b^2
  +4\cdot 4ab^3
  +{\bf 2}\cdot b^4
\] 
\end{ex}

In the last example we show how our construction generalizes the Euler characteristic of a determinantal hypersurface.
\begin{ex}\label{ex:Euler}
Consider the homogeneous cubic $f=(ac-b^2)d$ in $\PP^3$. It defines a surface that consists of $\PP^2$ and a cone over the second Veronese of $\PP^1$, both intersecting in the second Veronese of $\PP^1$. Thus the Euler characteristic equals:
\[\Chi(V(f))=\Chi(\PP^2)+1+\Chi(\PP^1)-\Chi(\PP^1)=4.\]
Hence the complement of $V(f)$ in $\PP^3$ has Euler characteristic equal to $\bf{0}$.

The cubic $f$ is the determinant of the matrix:
\[
\begin{pmatrix}
a&b&0\\
b&c&0\\
0&0&d\\
\end{pmatrix}.
\]
Inverting all (invertible) matrices in the above space of matrices we obtain a dominant map, i.e.~a parametrization of a variety of degree $a_1=1$.
We may also cut the above space of matrices with one (resp.~two, resp.~three) general hyperplanes. Then, inverting all (invertible) matrices we obtain a parametrization of a variety of degree $a_2=2$ (resp.~$a_3=2$, resp.~$a_4=1$).
We have:
\[\sum_{i=1}^4 (-1)^i a_i={\bf{0}}.\] 

In this example the tensor $T\in \CC^4\otimes \CC^3\otimes \CC^3$ and $\CC^4(T)$ is the space of symmetric $3\times 3$ matrices as above. We have:
\[\chi_T=\chi_T'=
  a^3
  +2\cdot 3a^2b
  +2\cdot 3ab^2
  +b^3
\]
\end{ex}

\section{Main results}\label{sec:def}
Let $\Mat_n$ be the space of $n\times n$ matrices that we identify with $\CC^n\otimes \CC^n$. We start by presenting two equivalent constructions of the variety of complete collineations $\COL_n$. 

For $i=1,\dots, n-1$ let $D_i\subset \PP(\Mat_n)$ be the projectivisation of the locus of matrices of rank at most $i$. We may consider a sequence of blow-ups:
\[\PP(\Mat_n)=:X_0\leftarrow X_1 \leftarrow\dots\leftarrow X_{n-2},\]
where in the $i$-th step we blow-up the strict transform of $D_i$. The variety of complete collineations $\COL_n:=X_{n-2}$. 

For the second construction let us consider the rational map:
\[\PP(\Mat_n)\dashrightarrow \PP(\CC^n\otimes \CC^n)\times \PP(\bigwedge^2\CC^n\otimes\bigwedge^2\CC^n)\times\dots\times\PP(\bigwedge^{n-1}\CC^n\otimes\bigwedge^{n-1}\CC^n),\]
where the map to the $i$-th component is given by taking all $i\times i$ minors of a matrix. The closure of the image of this map is $\COL_n$. We note that the map $\PP(\Mat_n)\dashrightarrow 
\PP(\bigwedge^{n-1}\CC^n\otimes\bigwedge^{n-1}\CC^n)$ may be identified with matrix inversion.

The second construction gives us natural projections $\pi_i:\COL_n\rightarrow \PP(\bigwedge^i\CC^n\otimes\bigwedge^i\CC^n)$ for $i=1,\dots, n-1$. In the Picard group of $\COL_n$ we obtain the divisors $L_i$ as pull-backs of hyperplanes $\pi_i^*(H)$. From now on we work in the vector space $\Pic_\QQ(\COL_n)$. The classes $L_i$ form a basis of that vector space.
\begin{rem}
We note that the vector space $\Pic_\QQ(\COL_n)$ has one more natural basis: the exceptional divisors of the blow-ups from the first construction together with the class of the pull-back of a hyperplane in $\Mat_n$, i.e.~$L_1$. 
\end{rem} 
Every cohomology class $S\in H^{2i}(\COL_n)$ gives a homogeneous polynomial $P_S$ of degree $i$ on $\Pic_\QQ(\COL_n)$ defined by:
\[P_S(D):=[SD^i]\in H^0(\COL_n,\QQ)\simeq \QQ.\]
\begin{defi}
Let $T\in\CC^a\otimes \CC^n\otimes \CC^n$ be such a tensor that the contraction $\CC^a(T)\subset \CC^n\otimes\CC^n$ contains a matrix of rank at least $n-2$. Let $d:=\dim \CC^a(T)$.
We define the polynomial on $\Pic_\QQ(\COL_n)$:
\[P_T(D):=S D^d \in H^0(\COL_n,\QQ)\simeq \QQ,\]
where $S$ is the Poincar\'e dual of  the strict transform of $\CC^a(T)$ by $\pi_1$.

In the basis $L_i$ we have:
\[P_T(\sum_{i=1}^{n-1} a_i L_i):=\sum_{b_1+\dots+b_{n-1}=d}{\binom{d}{b_1,\dots,b_{n-1}}}S\prod_{i=1}^{n-1} L_i^{b_i}.\]
We call the coefficients $T(b_1,\dots, b_{n-1}):=S\prod_{i=1}^{n-1} L_i^{b_i}$ the \emph{characteristic numbers} of the tensor $T$. 

Restricting to the line through $L_1$ and $L_{n-1}$ we obtain the \emph{chromatic polynomial} of $T$:
\[\Chi_T(aL_1+bL_{n-1}):=\sum_{i=0}^d \binom{d}{i} SL_1^i L_{n-1}^{d-i}.\]
\end{defi}

We note that by the definition the characteristic numbers are also the multidegree of the strict transform $S$ of $\CC^a(T)$ under the embedding 
\[S\subset \COL_n\subset \PP(\CC^n\otimes \CC^n)\times \PP(\bigwedge^2\CC^n\otimes\bigwedge^2\CC^n)\times\dots\times\PP(\bigwedge^{n-1}\CC^n\otimes\bigwedge^{n-1}\CC^n).\]

\begin{prop}
Any characteristic number $T(b_1,\dots,b_{n-1})$ is equal to the number of invertible matrices that satisfy $b_i$ polynomial conditions obtained by taking general linear combinations of $i\times i$ minors.
\end{prop}
\begin{proof}
All $L_i$ are base point free. We fix general representatives for the each of $b_i$ divisors equivalent to $L_i$. Intersecting them with the strict transform $S$ of $\CC^a(T)$ we obtain $T(b_1,\dots,b_{n-1})$ many points in $\COL_n$, all of them outside of the exceptional divisors of the blow-up and outside of the locus corresponding to matrices of rank $n-1$. Outside of the exceptional divisors, the blow-up $\COL_n\rightarrow \Mat_n$ is an isomorphism. The image of each divisor $L_i$ in $\Mat_n$ is the zero locus of a linear combination of $i\times i$ minors. Hence, the images of the divisors intersect $\CC^a(T)$ in $T(b_1,\dots,b_{n-1})$ many points corresponding to invertible matrices (and possibly a large subset of matrices of rank at most $n-1$). 
\end{proof}
\begin{cor}\label{cor:invMat}
Let $i:\PP(\Mat_n)\dashrightarrow \PP(\Mat_n)$
be the rational map inverting the matrices, which may be identified with the gradient of the determinant.
Let $\Gamma_T\subset \PP(\Mat_n)\times \PP(\Mat_n)$ be the restriction of the graph to $\PP(\CC^a(T))\subset\PP(\Mat_n)$, i.e.~the closure of pairs of matrices $([A],[A^{-1}])\in \PP(\Mat_n)\times \PP(\Mat_n)$ where $A\in \CC^a(T)$ and is invertible.

Then
\[\Chi_T(aL_1+bL_{n-1}):=\sum_{i=0}^d \binom{d}{i} m_ia^i b^{d-i},\]
where $(m_0,m_1,\dots)$ is the multidegree of $\Gamma_T$.
\end{cor}
\begin{proof}
This follows as in the previous proposition, once we notice that on the projectivisations the inversion of matrices may be identified with the gradient of the determinant and taking $(n-1)\times (n-1)$ minors.
\end{proof}
We see that we obtain the chromatic polynomial of $T$ by restricting the gradient of the determinant to $\CC^a(T)$ and looking at the multidegree of the graph. There is a closely related, but not the same in general, construction where we take the gradient of the restriction of the determinant.
\begin{defi}
Consider the polynomial $\det_{|\CC^a(T)}$ on $\PP(\CC^a(T))$. Its gradient defines a rational map:
\[\nabla(\det_{|\CC^a(T)}):\PP(\CC^a(T))\dashrightarrow \PP(\CC^a(T)^*).\]
Let $(m_0,m_1,\dots,m_d)$ be the multidegree of the graph of $\nabla\det_{|\CC^a(T)}$ considered as the subvariety of $\PP(\CC^a(T))\times \PP(\CC^a(T)^*)$.

We define the \emph{relative chromatic polynomial} of $T$ by:
\[\chi'_T(a,b):=\sum_{i=0}^d\binom{d}{i} m_ia^i b^{d-i}.\]
 
 Note that equivalently $\nabla(\det_{|\CC^a(T)})$ may be defined as a composition of the inclusion $\PP(\CC^a(T))\subset \PP(\Mat_n)$ with the gradient of the determinant and then with the projection $\PP(\Mat_n^*)\dashrightarrow \PP(\CC^a(T)^*)$ from $\PP(\CC^a(T)^\perp)\subset \PP(\Mat_n^*)$.
\end{defi}
\begin{lemma}
Let $h$ be a homogeneous polynomial on a vector space $V$. Let $L\subset V$ be a vector subspace not contained in the singular locus of $V(h)$. Consider two maps:
\[\nabla h: \PP(V)\dashrightarrow \PP(V^*),\]
\[\nabla (h_{|L}):\PP(L)\dashrightarrow \PP(L^*)=\PP(V^*/L^\perp).\]
Suppose that the map $\nabla (h_{|L})$ is generically finite.

The multidegree of the graph of $\nabla h$ restricted to $\PP(L)$ equals the multidegree of the graph of $\nabla (h_{|L})$ if and only if $\PP(L^\perp)$ is disjoint from $\overline{\nabla h(\PP(L))}$. If $L$ is general, then this condition is satisfied.
\end{lemma}
\begin{proof}
Note that the last entry $\mu$ of the multidegree of the graph of $\nabla h$ is the degree of $\overline{\nabla h(\PP(L))}$ (i.e.~the number of points one obtains after intersecting $\overline{\nabla h(\PP(L))}$ with $\dim \PP(L)$ many general hyperplanes) times the degree of the map $\nabla h$. On the other hand the last entry $\nu$ of the multidegree of the graph of $\nabla (h_{|L})$ is simply the degree of the map (as the closure of the image is the whole projective space). Hence, $\nu$ is the product of the degree of $\nabla h$ and the number of points that do not belong to $\PP(L^\perp)$ and are in the intersection of $\overline{\nabla h(\PP(L))}$ and $\dim \PP(L)$ many general hyperplanes that contain $\PP(L^\perp)$.

First, suppose that $\PP(L^\perp)\cap \overline{\nabla h(\PP(L))}\neq \emptyset$. Then, $\mu>\nu$ by \cite[Proposition 2.1]{AmendolaEtAl}.

Second, if $\PP(L^\perp)\cap \overline{\nabla h(\PP(L))}=\emptyset$, then the hyperplanes through $\PP(L^\perp)$ form a base-point free system on $\overline{\nabla h(\PP(L))}$ and hence, by Bertini theorem, the intersection consists of smooth points. Thus their number must be equal to the degree of $\overline{\nabla h(\PP(L))}$, and hence $\mu=\nu$. The proof that the other coefficients of polynomials are equal, is exactly the same, taking into account that by choosing a subspace of $L$, we obtain a subvariety of $\overline{\nabla h(\PP(L))}$. 

Last, by the results fo Teissier \cite{T1, T2} (see also \cite{kohn2020linear}), for general $L$, we know that $\PP(L^\perp)\cap \overline{\nabla h(\PP(L))}=\emptyset$. 
\end{proof}
\begin{cor}
We have $\chi_T=\chi'_T$ if and only if $\PP(\CC^a(T)^\perp)$ is disjoint from $\overline{\nabla\det (\PP(\CC^a(T)))}$.
\end{cor}
The next corollary is well-knwon. Indeed, by the more general results of Huh et al.~if $\CC^a(T)$ consists of diagonal matrices then the coefficients of $\chi_T'$ equal (up to binomial factors) the coefficients of the chromatic polynomial of the associated matroid \cite{June1}, and so do the coefficients of $\chi_T$ \cite{June2}. In particular, $\chi_T=\chi_T'$. As a direct proof is short and we find the fact very important, we present it below.
\begin{cor}
Suppose $\CC^a(T)$ consists of diagonal matrices. Then \[\PP(\CC^a(T)^\perp)\cap \overline{\nabla\det (\PP(\CC^a(T)))}=\emptyset\] and hence $\Chi_T=\Chi_T'$.
\end{cor}
\begin{proof}
Let $[a_0:\dots:a_n]\in \PP(\CC^a(T)^\perp)$. Without loss of generality we may assume $a_0=1$, $a_1\dots,a_k\neq 0$ and $a_{k+1}=\dots=a_n=0$. For contradiction let as assume that $[a_0:\dots:a_n]\in \overline{\nabla\det (\PP(\CC^a(T)))}$. This means that there exists a sequence of points $[b_{0,m}:\dots:b_{n,m}]\in \PP(\CC^a(T))$, such that:
\[[b_{0,m}^{-1}:\dots:b_{n,m}^{-1}]\rightarrow [a_0:\dots:a_n].\]
In particular, by rescaling, we may always assume $b_{0,m}=1$ and then $b_{i,m}\rightarrow a_i^{-1}$ for $i=1,\dots, k$. But then:
\[0=\sum_{i=0}^k b_{i,m}a_{i}\rightarrow k+1,\]
which is a contradiction.
\end{proof}
We next discuss why our setting of tensors, chromatic polynomials and characteristic numbers appears in different branches of mathematics, as shown in examples in Section \ref{sec:exm}.

By the results in \cite{June1, June2} we thus obtain the following corollary, which explains Example \ref{ex:graph}.
\begin{cor}
Let $G$ be a connected graph (or more generally a representable matroid) with $e$ edges. Let $A_G$ be the adjacency matrix of $G$.
The image of $A_G^t$ is a $j$-dimensional subspace of $\CC^{e}\simeq \Diag_{e}$, where we identify the ambient space with diagonal $e\times e$ matrices. This space gives rise to a tensor $T\in \CC^j\otimes \Diag_{e}\subset \CC^j\otimes \CC^{e}\otimes\CC^{e}$.

% Let $J$ be the $j$ dimensinal space that is the image of the (transpose of the)  Let $T\in \CC^j\otimes \CC^e\otimes\CC^e$ be the tensor represented by the map $J\rightarrow \CC^e\rightarrow \CC^e\otimes\CC^e$, where the last map is the diagonal embedding.

Then the coefficients of the reduced chromatic polynomial of $G$ are the characteristic numbers of $T$, which are also (up to binomial factors) the coefficients of $\chi'_T$. 
\end{cor}

Next, we provide relations to algebraic statistics, based on the results from \cite{StUh, ja1, ja2}. This explains Example \ref{ex:MLE}.
\begin{prop}
Consider a linear concentration model given by a space of symmetric matrices $L\subset S^2 V$. We may consider $L$ as a tensor $T\in \CC^{\dim L}\otimes S^2V\subset \CC^{\dim L}\otimes V\otimes V$.

The characteristic number $T(0,\dots,0,\dim L-1)$, i.e.~the last coefficient of the chromatic polynomial $\Chi_T$ is the degree of the associated statistical model. The last coefficient of the relative chromatic polynomial $\Chi_T'$ is the maximum likelihood degree of the statistical model. 
\end{prop}
\begin{proof}
The first statement is straightforward as by Corollary \ref{cor:invMat}, the number $T(0,\dots,0,\dim L-1)$ is the degree of the variety $L^{-1}$ obtained by inverting the matrices in $L$, which by definition is also the degree of the statistical model.

The second statement is proved in the following steps:
\begin{itemize}
\item The ML-degree equals the degree of the projection map $\pi$ with center $L^\perp$, restricted to $L^{-1}$  \cite{amendola2020maximum}.
\item The composition of the map $(\nabla \det)_{L}$ with $\pi$ is the gradient of the restriction of the determinant to $L$.
\item As the inversion of matrices map is birational, the ML-degree also equals the degree of the map $\nabla (\det_{|L})$. 
\item The last number is the last coefficient of $\Chi_T'$.
\end{itemize}
\end{proof}

%Next, basic on results of Huh et al.\cite{}, we show why and how the introduced chromatic polynomials coincide with chormatic polynomials of matroids and graphs.
%\begin{prop}
%Let $M$ be a representable (for simplicity connected) matroid, with a representation given by $n$ points in a vector space $V$. TBC

%In particular, let $G$ be a connected graph with adjacency matrix $A_G$ and chromatic polynomial $\Chi_G$. The image of $A_G^t$ is an $a$-dimensional subspace of $\CC^{|E}|\simeq \Diag_{|E|}$, where we identify the ambient space with diagonal $|E|\times |E|$ matrices. This space gives rise to a tensor $T\in \CC^a\otimes \Diag_{|E|}\subset \CC^a\otimes \CC^{|E|}\otimes\CC^{|E|}$. Then:
%$$\Chi_T(k)=\Chi_T'(k)=\Chi_G(-k).$$
%\end{prop}
%\begin{proof}
%\end{proof}

In the following proposition we explain why the characteristic numbers of tensors coincide with the characteristic numbers known in algebraic geometry. This is a very classical topic going back essentially to Schubert \cite{schubert1894allgemeine}.

Consider a linear system of quadrics $\PP(L)\subset S^2\CC^n$. The classical characteristic number is the answer to the following enumerative problem:

how many nondegenerate quadrics in $\PP(L)$ pass through $a$ general points and are tangent to $b$ general hyperplanes.

The system may be represented by a tensor $T\in \CC^{\dim L}\otimes S^2V\subset \CC^{\dim L}\otimes \CC^n\otimes\CC^n$. The following proposition explains Example \ref{ex:quad}.
\begin{prop}
The classical characteristic number for linear system of quadrics $\PP(L)$ is equal to the characteristic number $T(a,0,\dots,0,b)$.
\end{prop}
\begin{proof}
By Corollary \ref{cor:invMat} the number $T(a,0,\dots,0,b)$ is equal to the number of points one obtains by cutting the restriction $\Gamma'\subset \PP(S^2V)\times \PP(S^2V^*)$ of the graph of the invesion map to $\PP(L)$ with $a$ general hyperplanes in $\PP(S^2V)$ (times $\PP(S^2V^*)$) and $b$ general hyperplanes in $\PP(S^2V^*)$ (times $\PP(S^2V)$). 

We note that passing through a point is a linear condition on the space of quadrics $S^2V$. On the other hand, being tangent to a hyperplane is a linear condition on the space of dual quadrics $S^2V^*$.  Hence, the classical characteristic number is also the number of points $P$ we obtain by intersecting $\Gamma'$ with such linear conditions.

Note that this is not enough to conclude as neither passing through a general point, nor being tangent to a general hyperplane are general hyperplane conditions in $S^2V$ or $S^2V^*$. Further, although such sets of hyperplanes do not have base points (there is no quadric going through every point), we cannot apply classical results on base point free systems, as these are not formally linear systems. 

Still, for dimension reasons, all points $P$ must correspond to the set $S$ of pairs $(A,A^{-1})$ of invertible matrices. On this set, $GL(V)$ acts transitively, hence we may apply Kleiman's transitivity theorem to conclude that the intersection is transversal. Hence, the number of points must the the same as for general choice of hyperplanes.
\end{proof}

Finally, basing on \cite{June1} and \cite{dimca2003hypersurface}, we show how the relative chromatic polynomial is related to Euler characteristics of the determinantal locus. This explains Example \ref{ex:Euler}.
\begin{prop}
Let $L\subset \Mat_n$ and let $X=\PP(L)\cap V(\det)$. Let $T\subset \CC^{\dim L}\otimes \Mat_n=\CC^{\dim L}\otimes \CC^n\otimes\CC^n$ be the tensor representing $L$. Let $\chi'_T(a,b):=\sum_{i=0}^d\binom{d}{i} m_ia^i b^{d-i}$.
Then the Euler characteristics $\Chi(X)$ of $X$ equals:
\[n-\sum_{i=0}^d m_i.\]
\end{prop}
\begin{proof}
As the Euler characteristic is additive, it is enough to prove that the complement of $X$ has Euler characteristic $\sum_{i=0}^d m_i$. The numbers $\m_i$ form the multidegree of the gradient of the restriction of the determinant to $\PP(L)$, i.e.~the gradient of the polynomial that defines the hypersurface $X$. Hence, the $m_i$ are the mixed multiplicities defined in \cite[Definition 8]{June1} --- cf.~\cite[Remark 10]{June1}. The fact that the Euler characteristic of the complement is the signed sum of mixed multiplicities is stated as a corollary after \cite[Theorem 9]{June1}, based on the results of \cite{dimca2003hypersurface}.
\end{proof}

\section{Relation to tensor rank}\label{sec:rank}

Let us recall a classical tensor invariant, the \emph{rank}, defined as the
smallest $r$ so that $T$ may be written as a sum of $r$ rank one tensors.
The famous problem of determining the exponent of matrix multiplication was
shown by Strassen to be equivalent to determining the asymptotics of the rank of
the structure tensor of the matrix multipliation operator \cite{strassen1969gaussian, landsberg_2017}. 
For matrices, tensor rank is the familiar notion of
matrix rank, which may be efficiently computed with well known algorithms.
For tensors of three places or more, however, determining rank is a
difficult problem for which no efficient algorithm exists. 
Results for specific tensors often involve establishing upper
and lower bounds. For instance, the structure tensor of $3\times
3$ matrix multiplication, a tensor in $\CC^9\ot \CC^9\ot \CC^9$, is known to
have rank at least $19$ and at most $23$ \cite{Blaser2,laderman}.

As the characteristic numbers and chromatic polynomial of tensors are
potentially easier to determine than tensor rank, it is useful to understand any
relation with tensor rank. 
We compute the chromatic polynomials of generic tensors $T_{\text{gen},a,n,r}$ of rank
$r$ in $\CC^a\ot \CC^n\ot \CC^n$ for small values of $a$, $n$, and $r$. 
However, there is
significant redundancy in these numbers; specifically, a generic linear
restriction of a generic tensor of rank $r$ is still a generic tensor of rank
$r$, so we have
\[
  \bigg[\binom{a-1}{k} a^kb^{a-1-k}\bigg] \chi_{T_{\text{gen},a,n,r}} = 
  [b^{a-k-1}] \chi_{T_{\text{gen},a-k,n,r}}.
\]

Hence, we need only give the numbers $b_{a,n,r} = [b^{a-1}]
\chi_{T_{\text{gen},a,n,r}}$ in order to describe all such chromatic
polynomials.
\begin{rem}
As a subspace of matrices $L\subset \CC^n\otimes\CC^n$ moves to a special position, the pullback of $L$ to the variety of complete collineations breaks into several components. One of them is the strict transform. Still many of them may contribute to the intersection product with the divisors in the variety of complete collineations. As in our construction we intersect with base-point-free divisors, the effective divisors contribute in a nonnegative way. Thus special subspaces $L$ give us smaller (or equal) charactistic numbers than general ones. For general $L$ (thus for general tensors) we have explicit methods to compute the characteristic numbers \cite{ja2}. This means that if we take a general tensor of high enough rank $r$, we know $b_{a,n,r}=:b_{a,n}$. For example:
\[b_{3,3}=4,b_{4,4}=27,b_{5,5}=206,b_{6,6}=1760,b_{7,7}=16472,b_{8,8}=168007,b_{9,9}=1866790.\]

However, for small $r$ the number $b_{a,n,r}$ will be smaller. It is very interesting to see when the transition happens. Our numerical results show that, for $n=3,4,5,6,7,8$ this is respectively $r=4,6,7,9,11,13$.

The characteristic numbers of special tensors can also be computed using
theoretical methods. One of the recent succesful applications was a proof of a
conjecture by Sturmfels and Uhler \cite[Conjecture 2]{StUh} given in \cite{MRV}. This shows that it is possible to provide explicit examples of tensors, for which the characteristic numbers grow exponentially with respect to the dimension.
\end{rem}
\begin{center}
\begin{tabular}{rrrrrrrrrrrrr} \toprule
& & \multicolumn{11}{c}{$a$} \\ \cmidrule(l){3-13}
$n$ & $r$ & 1 & 2 & 3 & 4 & 5 & 6 & 7 & 8 & 9 & 10 & 11 \\ \midrule
\addlinespace[3mm]
2 & 2 & 1 & 1 &  &  &  &  &  &  &  &  & \\
 & 3 & 1 & 1 & 1 &  &  &  &  &  &  &  & \\
 & 4 & 1 & 1 & 1 & 1 &  &  &  &  &  &  & \\
 & 5 & 1 & 1 & 1 & 1 &  &  &  &  &  &  & \\
\addlinespace[3mm]
3 & 3 & 1 & 2 & 1 &  &  &  &  &  &  &  & \\
 & 4 & 1 & 2 & 4 & 4 &  &  &  &  &  &  & \\
 & 5 & 1 & 2 & 4 & 8 & 10 &  &  &  &  &  & \\
 & 6 & 1 & 2 & 4 & 8 & 10 & 8 &  &  &  &  & \\
 & 7 & 1 & 2 & 4 & 8 & 10 & 8 &  &  &  &  & \\
\addlinespace[3mm]
4 & 4 & 1 & 3 & 3 & 1 &  &  &  &  &  &  & \\
 & 5 & 1 & 3 & 9 & 17 & 11 &  &  &  &  &  & \\
 & 6 & 1 & 3 & 9 & 27 & 61 & 55 &  &  &  &  & \\
 & 7 & 1 & 3 & 9 & 27 & 61 & 103 & 105 &  &  &  & \\
 & 8 & 1 & 3 & 9 & 27 & 61 & 103 & 133 & 127 &  &  & \\
 & 9 & 1 & 3 & 9 & 27 & 61 & 103 & 133 & 143 &  &  & \\
 & 10 & 1 & 3 & 9 & 27 & 61 & 103 & 133 & 143 &  &  & \\
\addlinespace[3mm]
5 & 5 & 1 & 4 & 6 & 4 & 1 &  &  &  &  &  & \\
 & 6 & 1 & 4 & 16 & 44 & 56 & 26 &  &  &  &  & \\
 & 7 & 1 & 4 & 16 & 64 & 206 & 356 & 229 &  &  &  & \\
 & 8 & 1 & 4 & 16 & 64 & 206 & 524 & 964 & 786 &  &  & \\
 & 9 & 1 & 4 & 16 & 64 & 206 & 524 & 1076 & 1802 & 1700 &  & \\
 & 10 & 1 & 4 & 16 & 64 & 206 & 524 & 1076 & 1874 & 2906 & 3044 & \\
 & 11 & 1 & 4 & 16 & 64 & 206 & 524 & 1076 & 1874 & 2951 & 4374 & \\
 & 12 & 1 & 4 & 16 & 64 & 206 & 524 & 1076 & 1874 & 2951 & 4374 & \\
\addlinespace[3mm]
6 & 6 & 1 & 5 & 10 & 10 & 5 & 1 &  &  &  &  & \\
 & 7 & 1 & 5 & 25 & 90 & 170 & 157 & 57 &  &  &  & \\
 & 8 & 1 & 5 & 25 & 125 & 520 & 1312 & 1660 & 812 &  &  & \\
 & 9 & 1 & 5 & 25 & 125 & 520 & 1760 & 4600 & 7100 & 4429 &  & \\
 & 10 & 1 & 5 & 25 & 125 & 520 & 1760 & 4936 & 11672 & 19729 & 15073 & \\
 & 11 & 1 & 5 & 25 & 125 & 520 & 1760 & 4936 & 11912 & 25759 & 45513 & 41145\\
 & 12 & 1 & 5 & 25 & 125 & 520 & 1760 & 4936 & 11912 & 25924 & 52828 & 95078\\
 & 13 & 1 & 5 & 25 & 125 & 520 & 1760 & 4936 & 11912 & 25924 & 52828 & 101876\\
 & 14 & 1 & 5 & 25 & 125 & 520 & 1760 & 4936 & 11912 & 25924 & 52828 & 101876\\
\bottomrule
\end{tabular}
\end{center}

\bibliography{bibML}
\bibliographystyle{plain}
\end{document}